\pdfoutput=1 
\documentclass[12pt,english,12pt,draftcls,onecolumn]{IEEEtran}
\usepackage[T1]{fontenc}
\usepackage[latin9]{inputenc}
\usepackage{amsmath}
\makeatletter
\usepackage{multicol}
\usepackage{algorithm2e}
\usepackage{float}
\floatstyle{ruled}
\newfloat{algorithm}{tb}{lop}
\floatname{algorithm}{Algorithm}
\IEEEoverridecommandlockouts
\makeatother

\IEEEoverridecommandlockouts
\usepackage{graphicx}
\usepackage{subfigure}
\usepackage{fullpage}
\usepackage{epsfig}
\usepackage{latexsym}
\usepackage{amsmath,amsthm,amssymb}
\usepackage{color}
\usepackage{boxedminipage}
\usepackage{algorithm2e}
\def\G{\mathcal{G}}
\def\g{\tiny{\cal G}}
\def\V{\mathcal{V}}

\def\E{\mathcal{E}}

\def\I{\mathcal{I}}
\def\O{\mathcal{O}}
\def\T{\mathcal{T}}

\def\P{\phi}
\def\tr{{\mbox{\bf trace}}}
\def\dg{{\mbox{\bf deg}}}
\def\diag{{\mbox{\bf diag}}}
\newtheorem{theorem}{Theorem}[section]

\newtheorem{lemma}[theorem]{Lemma}
\newtheorem{definition}[theorem]{Definition}
\newtheorem{example}[theorem]{Example}
\newcommand{\bea}{\begin{eqnarray}}
\newcommand{\eea}{\end{eqnarray}}
\newcommand{\beas}{\begin{eqnarray*}}
\newcommand{\eeas}{\end{eqnarray*}}
\newcommand{\leftm}{\left[\begin{array}}
\newcommand{\rightm}{\end{array}\right]}

       {
       \footnotesize\normalfont\ttfamily\raggedright
       \setlength{\rightmargin}{\leftmargin}
       \setlength{\itemsep}{-12pt}
       \setlength{\parsep}{20pt}
       \begin{lrbox}{\@tempboxa}%
       \begin{minipage}{\linewidth-2\fboxsep}
       }%
       {
       \end{minipage}%
       \end{lrbox}%
       \fcolorbox{black}{shade}{\usebox{\@tempboxa}}\newline\newline
       }%
\begin{document}
\title{A Sieve Method for Consensus-type\\Network Tomography\thanks{The research of the authors was supported by AFOSR grant FA9550-09-1-0091.}}
\author{Marzieh Nabi-Abdolyousefi and Mehran Mesbahi\thanks{The authors are with the Department of Aeronautics and Astronautics, University of Washington, Seattle, WA 98195-2400 USA (emails: {mnabi+mesbahi}@uw.edu)}}
\maketitle
\begin{abstract}
In this note, we examine the problem of identifying the interaction geometry
among a known number of agents, adopting a consensus-type
algorithm for their coordination.
The proposed identification process is facilitated by
introducing ``ports'' for stimulating a subset of network vertices
via an appropriately defined interface and observing the network's response
at another set of vertices.
It is first noted that under the assumption of
controllability and observability of corresponding steered-and-observed network,
the proposed procedure identifies a number of important features of the
network using the spectrum of the graph Laplacian.
We then proceed to use degree-based graph reconstruction methods
to propose a sieve method for further characterization of the underlying network.
An example demonstrates the application of the proposed method.
%
%
\end{abstract}
\section{Introduction}
Physical sciences are often concerned with
inferring models and physical parameters from data.
Given a model for a physical phenomena, computing the data values
is often referred to as the {\em forward problem.}
On the other hand, in {\em inverse problems}, the objective is the construction,
validation, invalidation, or reconstruction of the {\em model} from a set of measurements
associated with the system.
%
%
%
%
Inverse problems arise in fields such as astronomy, geophysics,
medical imaging, remote sensing, ocean acoustic tomography, and non-destructive testing
\cite{Babaeizadeh+Brooks+Isaacson_2007,Marques+Drummond+Vasco_2003,Mueller+Siltanen+Isaacson_2002}.
%
Closer to the present work are the inverse problems
associated with electrical networks~\cite{Curtis+Morrow_2000},
and the celebrated ``Can one hear the shape of a drum?" which aims to characterize
a manifold via its spectra~\cite{kac_1966},
or more recently, ``Can one hear the shape of a graph?"~\cite{Gutkin_2001}.
%
In fact, in this paper, we address the inverse problem related to consensus-type
coordination algorithms.
%
Consensus-type algorithms have recently been employed
for analysis and synthesis of a host of distributed protocols and control strategies
in multi-agent systems, including, flocking, formation control, rendezvous,
and distributed estimation~\cite{Mehran+Magnus_2009}.

%
One of the key aspects of this class of protocols is
the strong dependency between the interaction and information-exchange
geometry among the multiple agents, on one hand, and
the dynamic properties that these systems exhibit, on the other.
%
Motivated by this dependency,
in our work, we consider the scenario where the interaction network is inside a ``black box,''
and that only certain ``boundary'' nodes in the network
can be influenced and subsequently observed.
The ``input'' boundary nodes are then used to stimulate the
network, whose response is subsequently observed at the ``output'' boundary nodes.
Using this setup, in our complementary work~\cite{Marzieh+Mehran+2010},
we have presented a node knockout procedure that aims to find the generating function
of the graph Laplacian from the observed input-output data.
Our focus in the present work, in the meantime,
is to reduce the search space for the identification of the network topology
by blending ideas from system identification, integer partitioning, and degree-based graph reconstruction.
%
%
The implicit contribution of our analysis is its ramifications for exact identification from
boundary nodes for networks  that have an embedded consensus-type algorithms for their operation,
including formation flying, distributed estimation, and mobile robotics.

Our notation and terminology are standard.\footnote{The main focus of this work is on undirected graphs. However, the
extension of some results to the directed case will be examined in subsequent works.}
%
We denote by $\G=(\V,\E)$ the undirected
simple graph with vertex set $\V$
and edge set $\E$, comprised of two-element subsets of $\V$; we use ``nodes'' or ``agents''
interchangeably with ``vertices.''
Two vertices $u,v \in \G$ are called adjacent if $\{u,v\} \in \E$.
For vertex $i$, $\dg\ i$ denotes the number of its adjacent vertices or neighbors.
The Laplacian matrix for the graph $\G$
is denoted by $L(\G)$.
%
Laplacian matrices are positive semi-definite whose spectrum will be ordered
as $0=\lambda_1(L(\G))\leq \lambda_2(L(\G)) \leq \ldots \leq \lambda_n(L(\G))$.
%
%
We use $\P_{\g}(s)$ to denote the characteristic polynomial of the graph Laplacian.
%
%
The cardinality of the set $\mathcal{H}$ will be denoted by $|\mathcal{H}|$;
${\cal O}(f(n))$ on the other hand denotes a function of $n$ that is bounded by some constant multiple of $f(n)$
for large values of $n$.
\section{Network Identification}\label{system_ID}
%
%
%
%
%
%
Consider the consensus protocol 
adopted by $n$-nodes, where $x_i$ is the state of the $i$-{th} node, e.g., its position, speed, heading, voltage, etc.,
evolves according to the sum of the differences between the
$i$-th node's state and its neighbors.
%
%
Next, let a group of agents $\I \subset \V$ with cardinality $|\I|=r_{\I}$,
``excite'' the underlying coordination protocol by injecting signals to the network,
with another set of agents $\O\in \V$, of cardinality $|\O|=r_\O$,
measuring the corresponding network response.
Hence, the original consensus protocol from node $i$'s perspective
assumes the form 
\begin{eqnarray}
\dot{x_i}(t)=\sum_{ \{i,j\} \in \, \E} (x_j(t)-x_i(t))+B_iu_i(t), 
\end{eqnarray}
where $B_i=\beta_i$ if $i\in \I$, and zero otherwise.
Without loss of generality, we can always assume that $\beta_i=1$ and
modify the control signal $u_i(t)$ as $\beta_i u_i(t)$ if necessary.
Adding the observation ports to this ``steered'' consensus,
and having $y_j(t)=x_j(t)$ when $j \in \O$,
we arrive at the compact form of an input-output linear time-invariant
system, 
%
\begin{equation}
\dot{x}(t)=A(\G)x(t)+Bu(t), \quad  y(t) =Cx(t),  \label{excited_agreement}
\end{equation}
where $A(\G)=-L(\G) \in \mathbf{R}^{n\times n}$,
$B \in \mathbf{R}^{n\times r_\I}$, and
$C \in \mathbf{R}^{r_{\O} \times n}$. 
%

%
Even though in general sets $\I$ and $\O$ can be distinct and contain more than one element,
for the convenience of our presentation, we will assume that they are
identical- and at times, assume that the resulting input-output system is in fact SISO.
%
The extension of the presented results to the case when $\I$ and $\O$ are
distinct will be discussed after introducing the basic setup and approach.
%
%

We now pose the {\em inverse problem} of graph-based coordination algorithms,
namely, the feasibility of identifying the spectral and structural
properties of the underlying network $\G$ via the data facilitated by the input-output ports
$\I$ and $\O$.
%
%
In order to implement this program, however,
we need to assume that: (1) the identification procedure has knowledge of the number
of agents in the network, and (2) the input/output sets $\I$ and $\O$
have been chosen such that the system described in (\ref{excited_agreement}) is
controllable and observable.
Although the first assumption is reasonable in general,
the second one requires more justification which we now provide.
In the trivial case when $\I=\V$ and $B$ is equal to the identity matrix,
the input-output consensus (\ref{excited_agreement})
is controllable and by duality, observable.
%
%
%
%
However, more generally, the controllability/observability of the network
from a subset of its boundary nodes, is less trivial, and more to the point,
not guaranteed for general graphs~\cite{Mehran+Magnus_2009}.
In the meantime, since we will need controllability and observability of the network
for its identifiability, we will rely on a topical conjecture in the algebraic graph theory community 
to the effect
that for large values of $n$, the ratio of graphs with $n$ nodes that are not controllable from any single node
to the total number of graphs on $n$ nodes approaches zero as $n \rightarrow \infty$~\cite{Godsil_conjecture_2009};
this phenomena is depicted in Fig.~\ref{input_output_exm}.
%
%
{\em In the present paper, we take the controllability and the observability
of the underlying graph from the input  and output nodes as our working assumption.}
In the meantime it is always convenient to know when the network is uncontrollable from a given node.
\begin{lemma}\label{controllability_check}
Let $G(s)=C(sI-A)^{-1}B$ as the input-output realization of (\ref{excited_agreement}).
The uncontrollable/unobservable eigenvalues of (\ref{excited_agreement}) will not appear in the
corresponding entry of $G(s)$. Specifically, $G(s)$ will be order $n-i$ polynomial for
the SISO case with $n$ agents and $i$ uncontrollable/unobservable eigenvalues.
\end{lemma}
\begin{proof}
Since the underlying graph is undirected, matrix $A(\G)$ is symmetric and there exists a unitary matrix $U$ and a real
nonnegative diagonal matrix $\Lambda=\diag(\lambda_1, \ldots, \lambda_n)$
such that $A(\G)=U\Lambda U^T$. The columns of $U$ are an orthonormal set of eigenvectors for $A(\G)$.
The corresponding diagonal entries of $\Lambda$ are the
eigenvalues of $A(\G)$~\cite{Roger+Charles+1985}. Therefore,
\begin{equation}
G(s)=C(sI-A(\G))^{-1}B=C(sI-U\Lambda U^T)^{-1}B = CU (sI-\Lambda)^{-1}U^TB 
\end{equation}
From PBH test, if the system (\ref{excited_agreement}) is not controllable, there is an eigenvector that
is orthogonal to $B$. Therefore for an arbitrary uncontrollable eigenvalue $\lambda_i$, the
$i$-th row of $U^T$ is orthogonal to $B$ and $\lambda_i$ will not appear in $(sI-\Lambda)^{-1}U^TB$.
An analogous argument works for the unobservable case.
\end{proof}
\subsection{System Identification}
We now consider various standard system identification procedures in the context of identifying the
spectra of the underlying graph Laplacian, and subsequently, gaining insights into
the interconnection structure that underscores the agents' coordinated behavior.
%
%

System identification methods are implemented via
sampling of the system (\ref{excited_agreement}) at
discrete time instances,\footnote{The system identification methods work based on data sampling
from the system. Since we aimed to identify the interaction geometry of the network, we originally considered a continuous system. Therefore, we need to discretize the system (\ref{excited_agreement}).}
$\delta, 2\delta, \ldots, k\delta, \ldots$, with $\delta > 0$,
assuming the form
\begin{eqnarray}
z(k+1)  = A_{d}z(k)+B_{d} v(k), \quad w(k)  =  C_{d}z(k),\label{eq:discrete}
\end{eqnarray}
where $z(k)= x(k \delta)$, $v(k) = u(k \delta)$, $w(k) = y(k \delta)$, $A_{d}=e^{\delta A}$,
$B_{d}=\left(\int_0^\delta e^{At} dt \right) B$, and $C_{d}=C$.\footnote{The notation $e^{A}$ for a square
matrix $A$ refers to its matrix exponential.}
In fact, the system identification process leads to a realization of the model
\begin{eqnarray}
\widetilde{z}(k+1)=\widetilde{A}_d\widetilde{z}(k)+\widetilde{B}_du(k), \quad {\widetilde w}(k)=\widetilde{C}_d\widetilde{z}(k), \label{ident_disc}
\end{eqnarray}
where $(\widetilde{A}_d, \widetilde{B}_d, \widetilde{C}_{d})$ is the realization of $({A_d}, {B_d}, {C_d})$ in (\ref{eq:discrete}).
The estimated system (\ref{ident_disc}), on the other hand, is equivalent to the continuous-time system
\begin{eqnarray}
\dot{\widetilde{x}}(t)=\widetilde{A}\widetilde{x}(t)+\widetilde{B}u(k), \quad y(t)=\widetilde{C}\widetilde{x}(t), \label{ident_cont}
\end{eqnarray}
with
$\widetilde{A}_{d}=e^{\delta \widetilde{A}},$ 
$\widetilde{B}_{d}=\left(\int_0^\delta e^{\widetilde{A}t} dt \right) \widetilde{B},$ 
and $\widetilde{C}_{d}=\widetilde{C};$ 
\noindent in this case, $\widetilde{A}=(1/\delta) \log_M \widetilde{A}_{d}$ where $\log_M$ denotes the matrix logarithm.
Since the system (\ref{ident_disc}) is a realization of the system (\ref{eq:discrete}),
it follows that the estimated triplet $(\widetilde{A}, \widetilde{B}, \widetilde{C})$
is a realization of $(A, B, C)$ in (\ref{excited_agreement}).
As a result, there exists a similarity transformation induced by the matrix $T$, such that
$\widetilde{A}=TAT^{-1}$, $\widetilde{B}=TB$, and $\widetilde{C}=CT^{-1}$.
%
%
%
In fact, in the controllable/observable case, eigenvalues of $\widetilde{A}_d$ are precisely matched with the eigenvalues of $A_d$.
Obtaining a zero as eigenvalue of $\widetilde{A}_d$, which is equivalent of obtaining $-\infty$ as the eigenvalue of $\widetilde{A}$, is
a sign of uncontrollable and/or unobservable mode in (\ref{excited_agreement}).\footnote{This follows from Lemma~\ref{controllability_check} since $-\infty$ will appear as zero in the corresponding entries of $P(s)$.}
%
For example in the identification procedure called {\em Iterative Prediction-Error Minimization Method},
the model (\ref{eq:discrete}) for every input $v_i$ and output $w_j$ can be represented as
${\bf A}(q)w_j(k)={\bf B}(q)v_i(k)$, where ${\bf A}(q) = 1+a_{1}q^{-1}+\dots+a_{n}q^{-n}$
and ${\bf B}(q) = b_{1}q^{-1}+\dots+b_{r_{\I}}q^{-r_{\I}}$.
%
%
%
{The unknown model parameters $\theta= [a_{1}, \dots, a_{n}, b_{1}, \dots, b_{r_{\I}}]$
can then be estimated by comparing the actual output
$w_j(k)$ and predicted output $\widetilde{w}_{ji}(k|k-1)$ using the mean-square minimization.
In this case, the output predictor is constructed as
$\widetilde{w}_{ji}(k|k-1)=[-w_j(k-n), \ldots, -w_j(1), v_i(k-r_{\I}), \ldots, v_i(1)]$.}
In yet another candidate system identification procedure, namely the {\em Subspace Identification Method},
the system (\ref{eq:discrete}) is approximated by another system in the form (\ref{ident_disc}),
using the state trajectory of the dynamic system that has been determined from input-output observations.
The Hankel matrix, which can be constructed from the gathered input-output data,
plays an important role in this method. By constructing the Hankel matrix, the
discrete time system matrices $\widetilde{A}_{d}$, $\widetilde{B}_{d}$, and $\widetilde{C}_{d}$ can
then be determined.
Subsequently, the continuous-time estimated matrices $\widetilde{A}$, $\widetilde{B}$, and $\widetilde{C}$ can
be identified;
see~\cite{Ljung_1} for
an extensive treatment of system identification methods.
%

In summary, an identification procedure such as the above two methods, implemented on
a controllable and observable steered-and-observed coordination protocol (\ref{excited_agreement}), leads to
a system realization whose state matrix is {similar} to the underlying graph Laplacian and in particular
sharing the same spectra and characteristic polynomial.
However, a distinct and fundamental issue in our setup is that having found a
matrix that is ``similar'' to the Laplacian of a network is far from having exact knowledge of the network
structure itself \cite{Marzieh+Mehran+2010}.
This observation motivates the following question: {\em to what extend does the knowledge of the
spectra of the graph, combined with the knowledge of the input-output matrices,
reduce the search space for the underlying interaction geometry?}
In this note, we explore this question using techniques based on integer partitioning and
degree-based graph reconstruction.
\section{Graph Characterization}
We first review qualitative characterization of the
underlying interconnection topology via its identified characteristic polynomial.
We then explore the possibility of reducing the search space for the underlying
network via the proposed {\em sieve method}.
%
\subsection{Graph Characterization via the Characteristic Polynomial}
Recall that via a system identification method,
the characteristic equation of the system (\ref{excited_agreement})
can be found as 
\begin{eqnarray}
\P_{\g}(s) &=& \det(sI-A(\G)) = s^n+a_1s^{n-1}+...+a_{n-1}s+a_n.   \vspace{-0.15in}           \label{charactristic_en}
\end{eqnarray}
%
%
Although the spectra of the graph Laplacian in general is insufficient to
form an explicit characterization of the underlying network, it leads
to a number of useful structural information about its geometry; we list a few:
(1) the value $({1}/{n}) \prod_{i=2}^n \lambda_i(\G)$ is the number of spanning trees in $\G$,
(2) one has $|\E|=1/2 \sum_{i=1}^{n}\lambda_i $, where $|\E|$ is the number of edges in the graph,
(3) if $a_{n-1}=n$, the underlying interconnection is a tree.
For a tree, the coefficient $a_{n-2}$ is also called the graph Wiener index
(the sum of all distances between distinct vertices of $\G$) \cite{Mohar+1991},
(4) if the associated graph is a tree, $a_k$ is the number of $k$-matching in the subdivision of $\G$
(where each edge of $\G$ is replaced by a path of length 2), 
(5) if the eigenvalues of $L(\G)$ are distinct, with $0, \lambda_2, \ldots, \lambda_r$ where
$\lambda_r> \ldots > \lambda_2>0$, define $\psi_\G(x)=(x-\lambda_2) \ldots (x-\lambda_r)$.
It is then well-known that $\psi_{\G^c}=(-1)^r\psi_\G(n-x)$, where the graph $\G^c$ is the complement of the graph $\G$.
The Hoffman number of the graph, $\mu(\G)$, can also be found as $\lambda_2 \lambda_3 \ldots \lambda_r /n$ whose properties and
applications have been studied in \cite{Tranishi_2003},
(6) let $\T$ be a tree with $n \geq 2$ vertices. If $\T$ has only one positive Laplacian eigenvalue with multiplicity one, then $\T$ is the star $K_{1,n-1}$~\cite{Tranishi_2003},
%
(7) let $\G$ be a connected graph with exactly three distinct Laplacian eigenvalues. Then the algebraic connectivity (the second smallest Laplacian eigenvalue) of $\G$ is equal to one if and only if $\G$ is a star of $K_{1,n-1}$ with $n \geq 3$,
(8) if $\G$ is a connected graph with integer Laplacian spectra, then $d(\G) \leq 2 \kappa(\G)$, where $d(\G)$ is the diameter of the graph and $\kappa(\G)=(1/n)\prod_{i=2}^n \lambda_i(\G)$.

Although the spectra of the Laplacian provides important insights, as noted above,
into the structural properties of the network,
we now proceed to explore the possibility of complete identification of the underlying
network using its graph spectra complemented with a sieve method.
%
%
\subsection{Graph Sieve}
In this section, we provide an overview of the graph sieve procedure--
that in conjunction with the identified Laplacian spectra--
leads to a more confined search for the network in the black box.
The essential ideas involve the judicious use of
integer partitioning algorithms and degree-based graph reconstruction.

%
%
%
Recall that with the standing assumption of $C=B^T$ in (\ref{excited_agreement}),
for the identified system matrices $(\widetilde{C},\widetilde{A},\widetilde{B})$,
after appropriate relabeling, the product $\widetilde{C}\widetilde{A}\widetilde{B}= CAB$
leads to the first $r\times r$ block partition of the matrix $L(\G)$. Notice that if
$B\neq C^T$, we still obtain $r^2$ entries  of the matrix $L(\G)$ which may not contain the diagonal entries.
The product $\widetilde{C}\widetilde{A}\widetilde{B}$ gives us the degree of the nodes
that are in common in both input and output sets and information on the minimum degree of other nodes.
Since the eigenvalues of the identified matrix $\widetilde{A}$ are identical to those
of $L(\G)$,
%
%
as the result of the identification process,
we have access to the sum of the degrees of all nodes in the network,
as well as the degrees of a subset of $r$-boundary nodes.
Let us define $\widetilde{R}$ as the set of nodes in $r$-boundary nodes which appears
in both $\I$, and $\O$ with $|\widetilde{R}|=\tilde{r}$.
Moreover, let 
\begin{eqnarray}
r_d = \sum_{v \in \tilde{R}} \dg\ v \quad \mbox{and} \quad s = \tr(\widetilde{A})-r_d. \label{s_formulation}
\end{eqnarray}
%
%
If $B=C^T$, then $r_d=\tr(\widetilde{C} \widetilde{A}\widetilde{B})$, and
the set $\widetilde{R}$ will be equal to $r$-boundary nodes.
We can then proceed to determine the degrees of $n-\tilde{r}$ remaining nodes,
or equivalently, partition the positive integer $s$ (\ref{s_formulation})
into $n-\tilde{r}$ integers, each assuming a value between
1 and $n-1$, and a lower bound for the degrees of $r-\tilde{r}$ nodes~\cite{dickson_1971}.
The possible values for the partitioning comes from the fact that
we expect the resulting graph be connected while respecting the bounds
on the maximum allowable node degrees.

Partitioning integers without constraints on the resulting partition
is often referred to as {\em unrestricted partitions}.
Restricted partitions, on the other hand,
are those with constraints on the largest value of the partition that is no greater
than a value of $K_U$,  or no smaller than $K_L$, or both.
Algorithms that generate unrestricted partitions can often be used to generate the
restricted ones by certain modifications.
Several such algorithms, dealing with unrestricted and restricted integer partitioning,
have been suggested in literature.
%
%
%
In the context of the graph realization using the proposed system identification
method, we proceed to use the algorithms in~\cite{comb_alg_Reingold}, 
in order to generate different possible sets of $n-r$ integers
between 1 and $n-1$, and $r-\tilde{r}$ nodes with specified
lower bounds on their degree such that their sum is $s$ (\ref{s_formulation}).
%
\subsection{Integer Partitioning Algorithms and Complexity Analysis}
Consider a degree sequence $\{d_1, d_2, \dots, d_{n-\tilde{r}}\}$ while $d_1+d_2+\dots+d_{n-\tilde{r}} = s$, with specified lower bound on $r-\tilde{r}$ of them.
Without loss of generality, assume that the first $\{d_1, \dots, d_{r-\tilde{r}}\}$ degrees
are lower bounded as
\begin{equation}
d_i \geq L_i \quad \mbox{for} \quad i=1, \dots, r-\tilde{r}. \label{bound}
\end{equation}
We are interested to find all possible partitioning of $s$ into
$n-\tilde{r}$ integers between $1$ and $n-1$ satisfying (\ref{bound}).
We have the following observation; see \cite{comb_2000}.
%
\begin{lemma}
Let the number of partitioning of $s$ into $n-\tilde{r}$
integers between $1$ and $n-1$ be denoted by $P_{n-\tilde{r}}(s)$.
Then
\begin{equation}
P_{n-\tilde{r}}(s) = P_{n-\tilde{r}-1}(s-1)+(n-\tilde{r})P_{n-\tilde{r}}(s-1). \label{num_partitions}
\end{equation}
\end{lemma}
%
The following algorithm, proposed in \cite{comb_alg_Reingold}
finds all partitioning of $s$ into $m= n-\tilde{r}$ integers between $1$ and $n-1$
satisfying (\ref{bound}). The partitioning of $s$ with $m$ components can be generated in increasing lexicographic order by starting with $d_1=d_2=\ldots=d_{m-1}=1,\ \ d_m=s-m+1$ and continuing as follows. To obtain the next partition from the first one, scan the elements from right to left, stopping at the right most $d_i$ such that $d_m-d_i \geq 2$. Replace $d_j$ by $d_i+1$ for $j=i, i+1, \ldots, m-1$ and then replace $d_m$ by $s-\sum_{j=1}^{m-1}d_j$. For example, if we have $s=12, \ \ m=5,$ and the partition $\{1,1,3,3,4\}$, we find that $4$ is greater by $2$ than the rightmost $1$, and so the next partition is $\{1,2,2,2,5\}$. When no element of the partition differs from the last by more than $1$, we are done.
%


\begin{algorithm}
\caption{Integer Partitioning}
\label{int_part}
\SetAlgoLined
$d_1=d_2=\ldots=d_{m-1}=1,\ \ d_m=s-m+1$ \\ \; $i=1$ \\ \;
\While {$i \neq 0$} {
\If {$\{d_1, d_2, \ldots, d_{\tilde{r}}\} \geq \{L_1, L_2, \ldots, L_{\tilde{r}}\}$ and $1 \leq d_i \leq n-1,\ \forall i$}
{output $\{d_1, d_2, \ldots, d_m\}$}
$i=m-1$\\ \;
\While {$d_m-d_i < 2\ \ $}    {
$i=i-1$
}
\If {$i \neq 0\ \ $ } {
\For {$j=m-1$ to $i$ by -1} {
$d_j=d_i+1$
}
}
$d_m=s-\sum_{j=1}^{m-1}d_j$  \;
}
\end{algorithm}
In the suggested algorithm the output size of each partitioning of $s$
into some arbitrary number of integers $m$, $P_m(s)$, is ${\cal O}(s)$.
This means that the total output size is ${\cal O}(sP(s))$. The approximate size of the number $P(s)$ is provided by the following
asymptotic formula, 
\begin{equation}
P(s) \sim \frac{1}{4 \sqrt{3}} \exp{\left(\pi \sqrt{\frac{2s}{3}}\right)}. \vspace{-0.10in} \nonumber
\end{equation}

In other words, $P(s)$ grows faster than any polynomial, but slower than any exponential function $Q(s)=c^s$.
%
However, in our application, we are interested in a subset of $P(s)$ which has a
specified size and satisfy certain constraints.
Specifically, the integer $s$ in (\ref{s_formulation})
is approximately  the number of edges in the graph.
For simple graphs
if ${\cal O}(|\E|)={\cal O}(n)$, then the upper bound for the proposed partitioning is ${\cal O}(ne^{\sqrt{n}})$, and
if ${\cal O}(|\E|)={\cal O}(n^2)$, the upper bound for the proposed partitioning is ${\cal O}(n^2e^{n})$.

Although the algorithm above leads to a possible
degree sequence for the underlying graph--- consistent with the
identification procedure-- we need an additional set of conditions
for ensuring that the obtained sequence in fact corresponds to that of
a graph.
\begin{definition} 
A {graphical sequence} is a list of nonnegative numbers that
is the degree sequence of some simple graph. A simple graph with degree sequence $\mathbf{d}$ is
said to realize $\mathbf{d}$.
\end{definition}
%
Our next step is therefore to characterize the necessary and sufficient
conditions for a set of integers to be graphical.
For this, we resort to the following result.
\begin{theorem}\cite{Douglas_2000}
\label{simple_graph_cond}
For $n>1$, an integer list $\mathbf{d}$ of size $n$ is graphical if and only if $\mathbf{d}'$
is graphical, where $\mathbf{d}'$ is obtained from $\mathbf{d}$
by deleting its largest element $\Delta$ and subtracting $1$
from its $\Delta$-th next largest elements.
The only $1$-element graphical sequence is $\mathbf{d}_1=\{0\}.$
\end{theorem}
\begin{example}
Consider a sequence $\{3, 2, 2, 2, 2\}$. Since the number of odd degree
nodes is odd, the sequence is not graphical.
Let us also construct the $\mathbf{d}'$ sequences described above,
as $\{1, 1, 1, 2\}$ and $\{ 1 \}$, which again verifies that this sequence is not graphical.
\end{example}

Since the maximum number of steps to check whether a sequence is graphical or not is $n$, the complexity
of this algorithm is ${\cal O}(n)$. Given the particular algorithmic means of generating a graphical sequence for
the required integer partitions,
as detailed above, we now consider the problem of
constructing graphs based on a graphical sequence.
\subsection{Degree Based Graph Construction Algorithms and Complexity Analysis}
Before describing the algorithm, we need to provide a few definitions.\footnote{The necessary definitions and algorithms have been discussed in \cite{kim+zoltan+e.a._09} and are briefly described here to complement the presentation.}
Let $A(i)$ denote the adjacency set of node $i$ defined
as
$$A(i)=\{a_k \, | \, a_k\in \V, a_k\geq i, \; \mbox{for all} \; k, \ 1\leq k \leq d_i \}.$$
The reduced degree sequence ${\mathbf d}^{'}|_{A(i)}$ is obtained after removing node $i$ with all its edges from $\G$.
We now define the ordering $\leq$ between two adjacency sets
of node $i$,
$A(i) = \{\ldots, a_k, \ldots\}$ and $B(i)=\{\ldots, b_k, \ldots\}$,
as $B(i) \leq A(i)$
if we have $b_k\leq a_k$ for all $1 \leq k \leq d_i$.
In this case we also say that $B(i)$ is "to the left" of $A(i)$. 
The next lemma introduces a sufficient condition for the sequence ${\mathbf d}^{'}|_{B(i)}$ to be graphical.
\begin{lemma}~\cite{kim+zoltan+e.a._09}\label{lemma4}
Let $\mathbf{d}=\{d_1, d_2, \ldots, d_n\}$ be a non-increasing graphical sequence, and let $A(i)$, $B(i)$ be two adjacency sets for some node $i \in \V$, such that $B(i) \leq A(i)$.
If the degree sequence reduced by $A(i)$ (that is ${\mathbf d}^{'}|_{A(i)}$) is graphical,
then the degree sequence reduced by $B(i)$ (that is ${\mathbf d}^{'}|_{B(i)}$) is also graphical.
\end{lemma}
The above lemma guarantees preservation of ``graphicality''
for all adjacency sets to the left of a graphical one.
Now consider a graphical degree sequence $\mathbf{d}$ on $n$ nodes obtained
from previously discussed integer partitioning approach.
From the identity $\widetilde{C}\widetilde{A}\widetilde{B}=CAB$,
as we discussed in \S~\ref{system_ID}, $r_{\I} r_{\O}/{2}$ entries of the system matrix $A(\G)$ are known;\footnote{In the case where $C= B^T$, $r^2/2$ entries of $\widetilde{C}\widetilde{A}\widetilde{B}$ are known.}
define these set of edges as being ``pre-determined'' in the graph which cannot
be repeated again.
Put these connections in the forbidden set $X(\mathbf{d})$.
The Algorithm~\ref{degree_graph_const} describes how we can
construct all possible graphs avoiding the edges in $X(\mathbf{d})$.
\begin{algorithm}
Given a graphical sequence $\{d_1 \leq d_2 \leq \ldots \leq d_n \leq 1\}$.

\begin{itemize}
\item[I.]{Define the rightmost adjacency set $A_R(i)$ containing the $d_i$ largest index nodes different from $i$. Let us also define $X(i)=\{j \in \V, j\neq i\ \mbox{s.t.}\ \{i,j\}\notin \E \}$ as the forbidden neighbors of node $i$. Note that $X(i)$ originally might contain some nodes according the forbidden set $X(\mathbf{d})$. Create the set $A_R(1)$ and $X(1)$ for node $1:$ connect node $1$ to $n$ (this never breaks graphicality). Set $X(1)=\{n\}$. Define the new sequence ${\mathbf d}^{'}=\{d_1-|X(1)|, d_2, \ldots, d_n\}$.
Let $k=n-1$.}
\begin{itemize}
\item[I.1.] {Connect another edge
of $1$ to $k$. Run the graphicality test in Theorem~\ref{simple_graph_cond}.} \label{I.1}

\item[I.2.] If this test fails, set $k=k-1$; repeat I.1. 

\item[I.3.] If the test passes, keep (save) the connection, add the node $k$ to the forbidden set $X(1)$ and update the degree sequence ${\mathbf d}^{'}=\{d_1-|X(1)|, d_2, \ldots, d_n\}$, set $k=k-1$, and if $i$ has edges left, repeat from I.1. 

\end{itemize}
\item[II.] Create the set $A(\mathbf{d})$ of all adjacency sets of node $1$ that are colexicographically smaller than $A_R(1)$ and preserve graphicality, i.e.,
\begin{eqnarray*}
A(\mathbf{d}) = \left\{ A(1)= \{a_1, \ldots, d_{d_1}\}, a_i \in \V \, | \, {A(1)} <_{CL} A_R(1), {\mathbf d}^{'}|_{A(1)}\; \mbox{is graphical} \right\}. 
\nonumber
\end{eqnarray*}
where the operation $<_{CL}$ means a colexicographic order between two sets.
\item[III.]{For every $A(1) \in A(\mathbf{d})$ create all graphs from the corresponding graph realization of
$\mathbf{d}^{'}|_{A(1)}$
using this algorithm, where $\mathbf{d}^{'}_{A(1)}$ is the sequence reduced by $A(1)$.}
\end{itemize}
\caption{Degree Based Graph Construction}
\label{degree_graph_const}
\end{algorithm}
%
When constructing $A(\mathbf{d})$, checking graphicality is only needed for those
adjacency sets which are incomparable by the ordering relationship to any of the current
elements of $A(\mathbf{d})$; for the remaining sets graphicality is guaranteed by Lemma~\ref{lemma4}.

The total number of graphs that the above algorithm produces is $\Pi_{i} (d_{i}!)$.
Of course, this procedure makes sense for degree sequences
$\mathbf{d}$ for which there is only a small number of labeled
graphs realizing it.
An upper bound on the worst case
complexity $C_{\mathbf{d}}$ of the algorithm for constructing a
sample from a given degree sequence $\mathbf{d}$ is
$C_{\mathbf{d}} \leq {\cal O}(n |\E|)$, with $|\E|$ being the
number of edges in the graph.
For simple connected graphs, the maximum possible
number of edges is ${\cal O}(n^2)$, and the minimum possible number is ${\cal O}(n)$.
If ${\cal O}(|\E|)={\cal O}(n)$, then $C_{\mathbf{d}} \leq {\cal O}(n^2)$, and
if ${\cal O}(|\E|)={\cal O}(n^2)$, then $C_{\mathbf{d}}\leq {\cal O}(n^3)$, which
is an upper bound, independent of the degree sequence~\cite{Genio+Kim+Toro+Bassler+2010}.
An
algorithm to realize 
graphical degree sequences of directed graphs has been
studied in \cite{Erdos+Miklos+Toro+2010}
which can be use to extend these results to simple directed graphs.
%
%
%
%
%
%

Our sieve method for confining our search for the
underlying graph following the system identification of \S~\ref{system_ID},
thus involves:
\begin{enumerate}
\item perform an integer partition on the value $s$ (\ref{s_formulation}), keeping the partitions that lead to
a graphical sequence; let $\bf{G}$ denote the set of all graphs that remain after this first stage of the sieve.
\item construct a candidate connected graph in $\bf {G}$  that is consistent with the matrix $\widetilde{C}\widetilde{A}\widetilde{B}$ and satisfy the given degree sequence,
%
%
%
\item compare the Laplacian eigenvalues of the constructed graphs
with the roots of the characteristic polynomial (\ref{charactristic_en}) and discard
the inconsistent graphs.
\end{enumerate}
The candidate graphs for the underlying network topology are now
among the ones that remain after this three step sieve. 
We note that the sieve method is guaranteed to reduce the search space for the graph structure
by at least a factor of $2^{n}$, a bound that is obtained from the bound on the number
of permissible degree-based integer partitioning facilitated by the network system identification.
%
An example for this procedure is given next.
%
\section{An Example}

Our goal in this example is to gather information
on the graph $\G$ shown in Figure \ref{simple_graph1}(a) using the system identification procedure.
%
Using nodes $1$, $2$, and $3$ as the input-output nodes
in
(\ref{excited_agreement}),
we obtain
$\P_{\G}(s)=s^6+220s^5+190s^4+804s^3 +1664s^2+1344s$.
Since the polynomial $\P_{G}(s)$ has just one zero root,
the underlying graph is connected. Moreover, the graph is not a tree due to the
fact that $a_{n-1}\neq 6$. The graph has $11$ edges and $224$ spanning trees.
We also obtain the estimation matrices $\widetilde{A}, \ \widetilde{B}$, and $\widetilde{C}$ to be
\begin{eqnarray}
\widetilde{A} &=& \left[ \begin{array}{cccccc}
   -2.0913 &   9.5928  &  -3.2434  &  -1.9321 &  -8.3119 &  -3.6056 \\
   -0.7817 &  -2.6237 &  -2.3217 &  -3.3284 &  -2.1037  &  0.8747\\
    0.2080  & -0.1114  & -3.4715  &  0.8443  &  0.0218  & -0.2108 \\
   -0.3304 &  -0.1269  &  0.5219  & -3.1500  &  0.0999  &  0.3707 \\
   -0.6406 &   2.3385 &  -3.4079 &  -4.7957  & -7.1871  &  0.6395 \\
   -1.8704  &  9.1263  & -2.8730 &  -1.3239 &  -7.4520  & -3.4763\\
 \end{array} \right] \nonumber\\
\widetilde{B} &=& \left[ \begin{array}{ccc}
   -9.7903 & 13.7604  & -4.0211 \\
    1.6650  &  0.0982  & -1.6752\\
    0.6348  &  0.1419  & -0.3147\\
    0.4699  & -0.3308 &  -0.4116\\
    0.3430   & 2.7623  & -3.1201\\
   -9.3019 & 12.8359 &  -3.5526\\
 \end{array} \right] \nonumber\\
\widetilde{C} &=& \left[ \begin{array}{cccccc}
  -28.1542 &   2.3826  &  0.3547 &   0.8880 &   0.6434 &  30.0441 \\
  -28.2311 &   2.0005 &   0.0729 &   0.2781  &  0.9634  & 30.1259\\
  -28.2455  &  2.5397 &  -0.5074 &  -0.5056  &  0.5157  & 30.1419\\
 \end{array} \right]. \nonumber
\end{eqnarray}
And also
\begin{equation}
\widetilde{C}\widetilde{A}\widetilde{B} =\left[ \begin{array}{ccc}

   -3.0000  &  1.0000 &  -0.0000 \\
    1.0000  & -4.0000  &  1.0000\\
    0.0000  &  1.0000 & -3.0000\\
 \end{array} \right]. \nonumber
\end{equation}
Since the diagonal of the matrix $\widetilde{C}\widetilde{A}\widetilde{B}$
is $[\, -3, -4, -3 \, ]^{T}$,
%
$d_1=3$, $d_2=4$, and $d_3=3$.
In the meantime, the sum of degrees of the remaining nodes is $12$.
The possible integer partitions for the remaining three nodes such that the sum is $12$
and each degree is less than $6$ will be $\{5,5,2 \}$, $\{5, 4, 3 \}$, and $\{4, 4, 4 \}$.
Therefore, the set of the possible degrees sequences is comprised of
$\{3, 4, 3, 5, 5, 2\}$, $\{3, 4, 3, 5, 4, 3\}$, and $\{3, 4, 3, 4, 4, 4\}$.
According to the Theorem \ref{simple_graph_cond}, three set of integer partitioning are realization of some graphs.

%
Next, we construct the graphs with the three candidate degree sequences by
implementing the algorithm in \cite{kim+zoltan+e.a._09} and constructing
all connected graphs consistent with $\widetilde{C}\widetilde{A}\widetilde{B}$
%
, and with degree sequences $\{3,4,3,5,5,2\}$, $\{3, 4, 3, 5, 4, 3\}$, and $\{3,4,3,4,4,4 \}$.
The first degree sequence, $\{3,4,3,5,5,2\}$, and the degree sequence $\{3, 4, 3, 5, 4, 3\}$ are realizations of two graphs while the third degree sequence is a realization of two graphs. All
these graphs satisfy the constraint imposed by $\widetilde{C}\widetilde{A}\widetilde{B}$.
%
%
By comparing the Laplacian spectra for the corresponding
$4$ graphs with the roots of the identified characteristic polynomial,
%
we can thereby identify the original graph.
%
Three candidates for the constructed graph
are depicted in Fig.~\ref{simple_graph1}(b)-(d).
%
\section{Conclusion}
In this paper, we introduced a network identification
scheme which involves the excitation and observation
of nodes running consensus-type coordination protocols.
Starting with the number of vertices in the network as a known parameter, as well as the
controllability and observability of the resulting steered-and-observed network,
the proposed procedure strives to collect pertinent information on the topology of the underlying graph.
In this direction, we examined the applications of spectral characterization of graphs as well as a sieve method
that is based on integer partitioning algorithms and feasible graphical sequences.
%
\vspace{-.1in}
\bibliographystyle{IEEEtran}
\bibliography{myref}

\newpage

\begin{figure}[!t]
\begin{center}
\scalebox{1.1}{\includegraphics{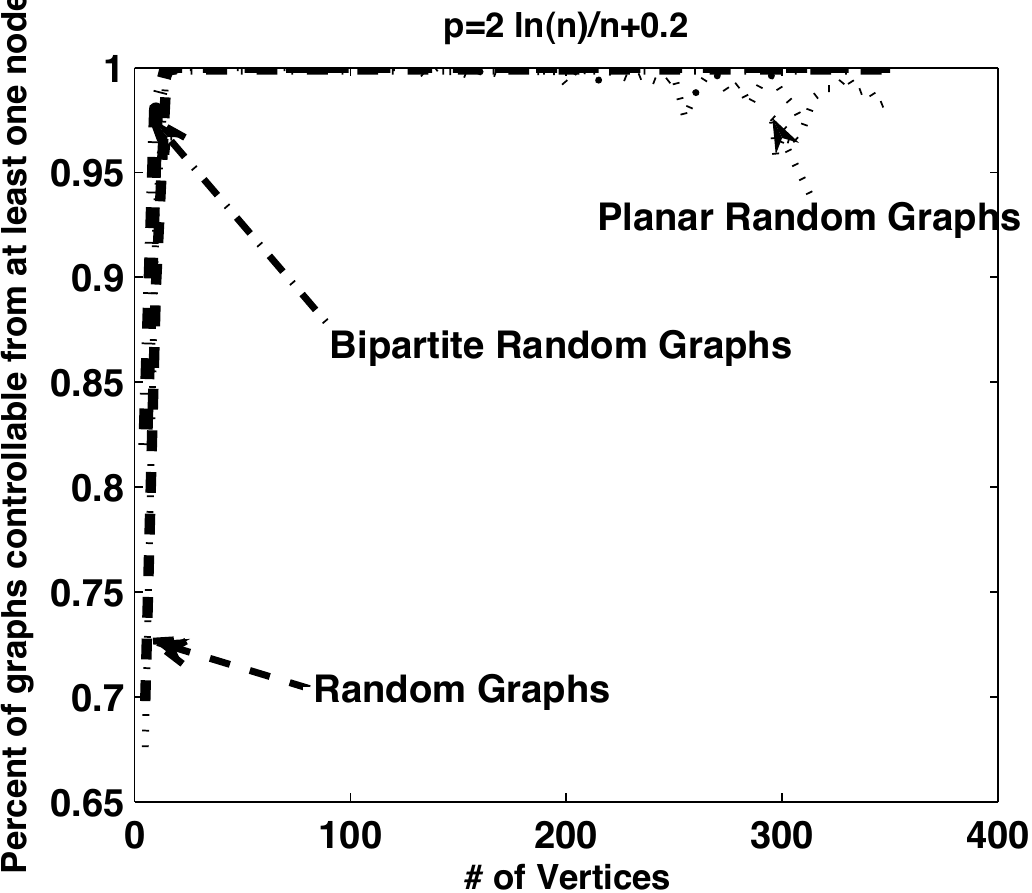}}
\caption{The percentage of random planar graphs that are controllable from at least one node} \label{input_output_exm}
\end{center}
\end{figure}

\clearpage

\begin{figure}[!t]
\centering
\subfigure[]{
   \scalebox{0.7}{\includegraphics{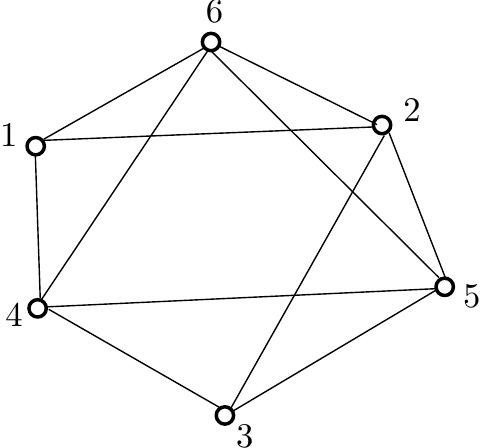}}}
\subfigure[]{
   \scalebox{.7}{\includegraphics{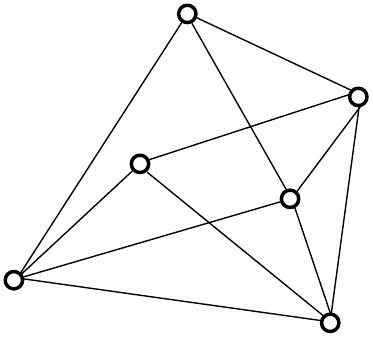}}}
\subfigure[]{
   \scalebox{.78}{\includegraphics{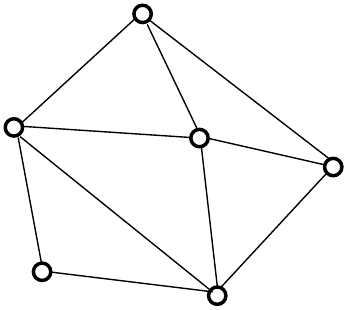}}}
\subfigure[]{
   \scalebox{.82}{\includegraphics{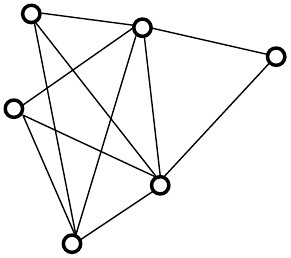}}}
   \caption{(a) a simple graph on 6 nodes, (b) a candidate graph
     constructed with degree sequence $\{3,4,3,4,4,4\}$, (c) a
     candidate graph constructed with degree sequence
     $\{3,4,3,5,4,3\}$, (d) a candidate graph constructed with degree sequence
     $\{3,4,3,5,5,2\}$.}  \label{simple_graph1}
\end{figure}
\newpage

\end{document}